\documentclass{article}
\usepackage[utf8]{inputenc}
\usepackage[T1]{fontenc}
\usepackage{amsmath}
\usepackage{amsthm}
\usepackage{amssymb}
\usepackage{esint}
\PassOptionsToPackage{normalem}{ulem}
\usepackage{ulem}
\usepackage{lmodern}
\usepackage[english]{babel}
\usepackage{mathtools}
\usepackage{graphicx} 
\usepackage{subfigure}
\usepackage{graphicx,float}
\usepackage{mathrsfs}
\usepackage{pgf,tikz}
\usepackage{mathrsfs}
\usetikzlibrary{arrows}

\newtheorem{theorem}{Theorem}

\newtheorem{remark}{Remark}

\newtheorem{example}{Example}

\makeatletter
  \theoremstyle{definition}
  \theoremstyle{remark}

\makeatother

\usepackage{babel}

\title{Saddle-Node Bifurcation and Homoclinic Persistence in  AFM with Periodic Forcing}
\author{Alexander Gutierrez G.\footnote{Universidad Tecnológica de Pereira (UTP), alexguti@utp.edu.co}, Daniel Cortés Z.\footnote{Universidad Tecnológica de Pereira (UTP), danielcorteszapata@utp.edu.co}, Diego A. Castro G.\footnote{Universidad Tecnológica de Pereira (UTP), xandercastro@utp.edu.co}}
\date{}

\begin{document}

\maketitle

{\bf Abstract}
We study the dynamics of an Atomic Force Microscope (AFM) model, under the Lennard-Jones force with non-linear damping, and harmonic forcing. We establish the bifurcation diagrams for equilibria in a conservative system. Particularly, we present conditions that guarantee the local existence of saddle-node bifurcations. By using the Melnikov method, the region in the space parameters where the persistence of homoclinic orbits is determined in a non-conservative system.


\section{Introduction}

The Atomic Force Microscopes (AFMs), were created  in 1986 by Bining, et. al, \cite {Quate}. They are based on the tunneling microscope and the needle profilometer principles. Generally, AFMs  measure the interactions between particles by allowing the nanoscale study of the surfaces for different materials, \cite{aplicacion4, aplicacion3, Morita}. In fact a wide variety of applications in analysis of pharmaceutical products, the study of the properties of fluids and  fluids in cellular detection, the medicine studies,  among others can be found in \cite{Bowen, Bru, aplicacion5, Alan}.

The  model is presented in \cite {Ashhab2, Ashhab}, where the authors study the interaction between the sample and the device's tip, see figure \ref{fig:AFM}. The associated differential equation is:

\begin{equation}\label{eq:nnn1}
     \ddot{y}+\frac{C}{(y+a)^3}\dot{y}+y =\frac{b_1}{(y+a)^8}-\frac{b_2}{(y+a)^2}+ f(t).
\end{equation}

where $b_1,b_2$ and $a$ are positive constants and $f$ is a continuous function $T$-periodic with zero average, that is, $\bar{f} = \frac{1}{T} \int_{0}^{T} f(t)\,dt = 0 $. The right hand side 

\[
F_{LJ}:=\frac{b_1}{(y+a)^8}-\frac{b_2}{(y+a)^2},
\]

is known as the Lennard-Jones force, which can be considered as a simple mathematical model to explain the interaction between a pair of neutral atoms or molecules; see \cite{Jensen,Lennard-Jones}  for the standard formulation. The first term describes the short-range repulsive force due to overlapping electron orbits so-called Pauli repulsion, whereas the second term simulates the long-range attraction due to van der Waals forces. This is a special case of the wider family of Mie forces
\[
F_{n,m}(x)=\frac{A}{x^n}-\frac{B}{x^m},
\]
where $n,m$ are positive integers with $n>m$, also known as the $n-m$ Lennard-Jones force, see \cite{Brush}. On the other hand, the dissipative term of \eqref{eq:nnn1}:
\[
F_r=\frac{C}{(y+a)^3}\dot{y},
\]
is associated with a damping force of  compression squeeze-film type. In specialized literature,  compression film type damping can be considered as the most common and dominant dissipation in different mechanisms, (see \cite{Mohammad, Amortiguamiento} and their bibliography).\\

\begin{figure}
    \centering
	\includegraphics[scale=0.5]{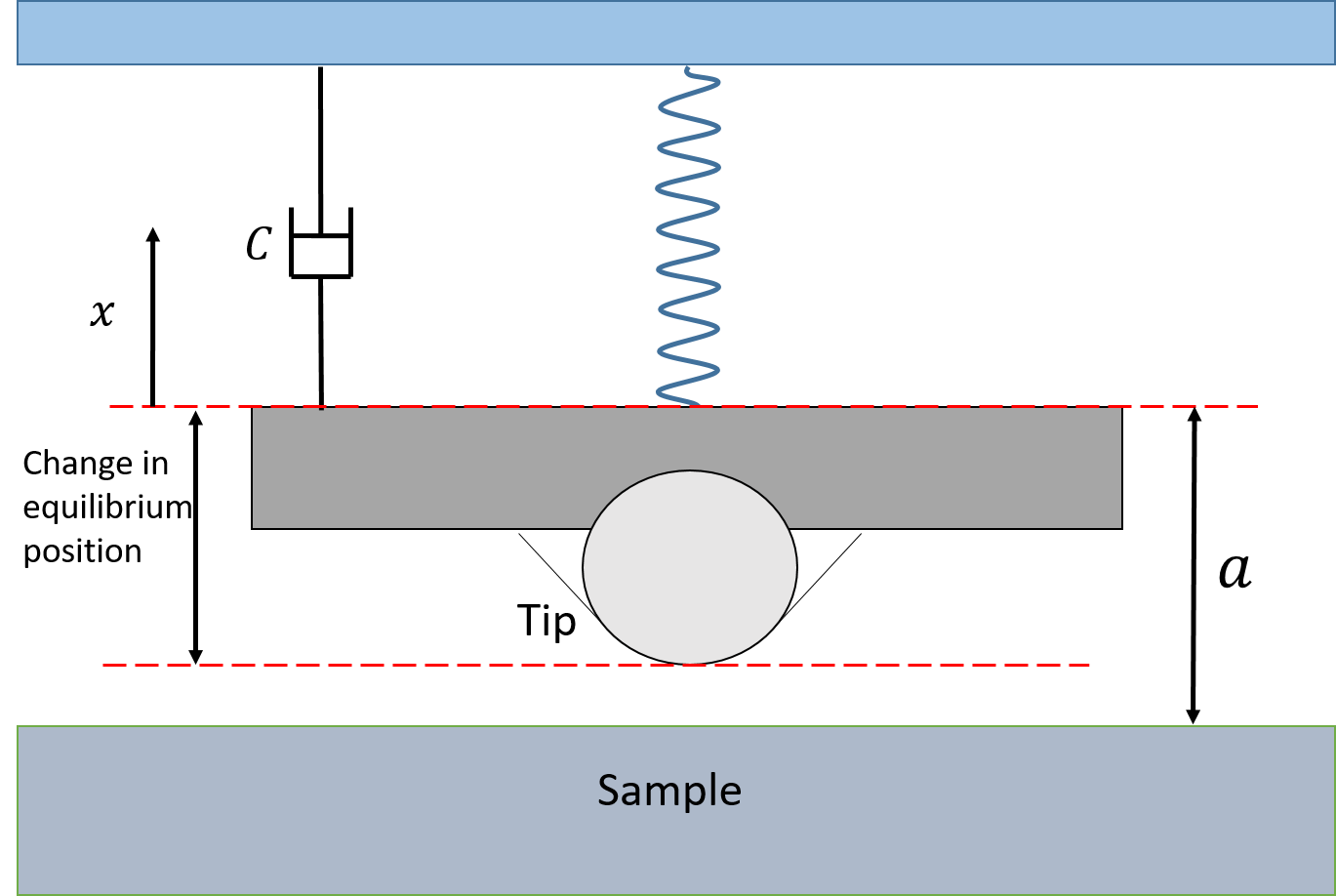}
	\caption{Mechanical model associated with the AFM's devices.}\label{fig:AFM}
\end{figure}

For the conservative system, two main results were obtained, Theorems \ref{teo:claseqv} and \ref{teo:bif}, where we  establish analytically the bifurcation diagram of the equilibria for specific regions with the involved  parameters  in contrast to the one obtained in  \cite{aplicacion7}. In particular, Theorem \ref{teo:bif} proves the local existence of two saddle-node bifurcations that can be related to hysteresis phenomenon, see for example \cite{hysteresis1, hysterisis2}.

 In the non-conservative system, we present as a  main result, Theorem \ref{teo:melnivok}, which gives a thorough and rigorous condition for the persistence of  homoclinic orbit  when the external forcing is of the form $ f(t)=B \cos (\Omega t)$. The condition found  relates the amplitude of the external forcing $B$ with the damping constant $C$, which in practice can be used to prevent the AFM device from becoming decalibrate.

This article is structured in the following way: this first section as an introduction, section two is dedicated to prove the main results in conservative system, and section three contains the proof for the main result of the non-conservative system along with some illustrative examples.

\section{Bifurcation Diagrams}

With the change of variable  $x=y+a$,  \eqref{eq:nnn1} is rewritten as 
\begin{equation}\label{eq:prinmod}
\ddot{x}=m(x)+a+\epsilon\bigg( f(t)-\frac{C}{x^3}\dot{x}  \bigg),
\end{equation}
where $m(x)=\frac{b_1}{x^8}-x-\frac{b_2}{x ^ 2}$ is the total force acting over the system, which is a combination of the Lennard-Jones force and the restoring force of the oscillator.  The change of the singularity from $-a $ to $0$ will facilitate the study of the bifurcation diagram for equilibria in the conservative system ($\epsilon=0$). Note that the classification of the equilibrium solutions of \eqref{eq:prinmod} plays an important role when the full equation is studied. We now  describe some properties of the function  $m(x)$: 
\begin{align*}
\lim_{x\to 0^+} m(x)&=\infty, & \lim_{x\to\infty} \dfrac{m(x)}{x}=-1,
\end{align*}
moreover $m$  has only one positive root and a direct analysis provides a critical value
\begin{equation}
b_1 ^* = \frac{4}{27}b_2^3 \label{b1},
\end{equation}
such that:
\begin{itemize}
	\item [i)]If $b_1> b_1^*$, then $m(x)$ is decreasing.
	\item [ii)] If $b_1=b_1^*$, then $m(x)$ is non-increasing and has an inflection point  in  $x_c=(\frac{4}{3} b_2)^{1/3}$.
	\item [iii)]  Finally, if  $b_1<b_1^*$, then  $m(x)$ has a local maximum (resp. minimum)   in $x_r$ (resp.$x_l$)  and $m(x_r),m(x_l)<0$.
\end{itemize}
Therefore, the equilibria set  $\mathcal{G}=\{x\in\mathbb{R}^+:\,m(x)+a=0\}$ is finite, not empty, and the number of equilibria depends on the parameter $a$. Figure \ref{fig:funcioneme}, shows the possible variants of the $m$ function in terms of $b_1$, $b_2$ and $a$.\\

\begin{figure}
	\subfigure[$m$ is decreasing monotone if  $b_1> b_1^*$ ]{\includegraphics[scale=0.4]{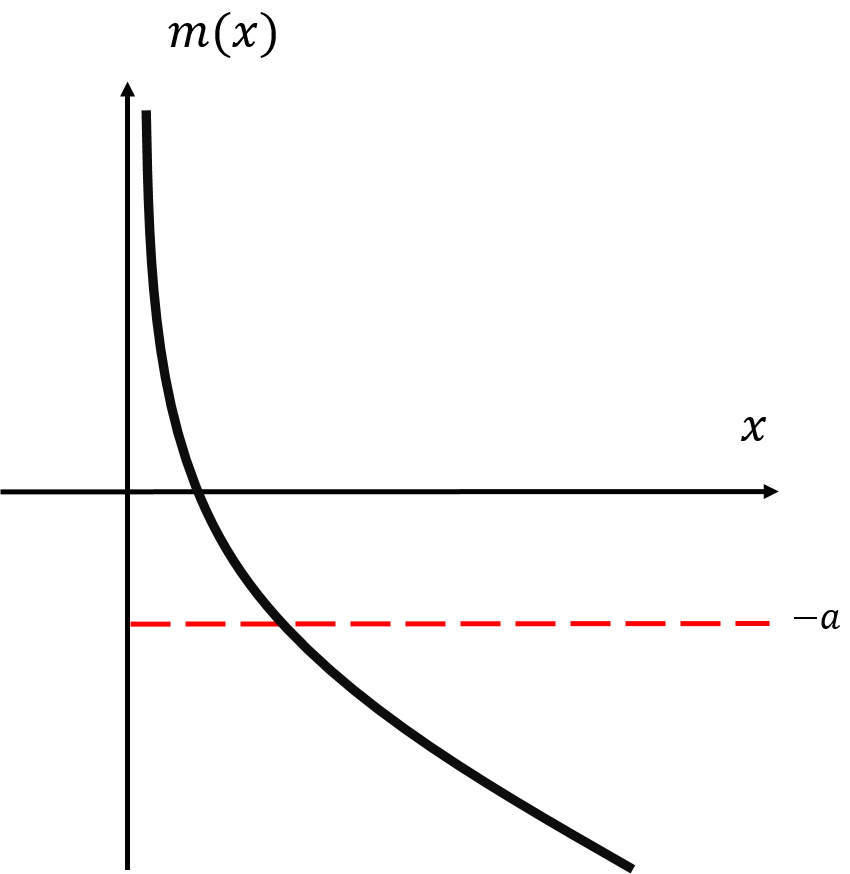}}\subfigure[ $m$  has a maximum and a local minimum, if $b_1< b_1^*$. ]{\includegraphics[scale=0.4]{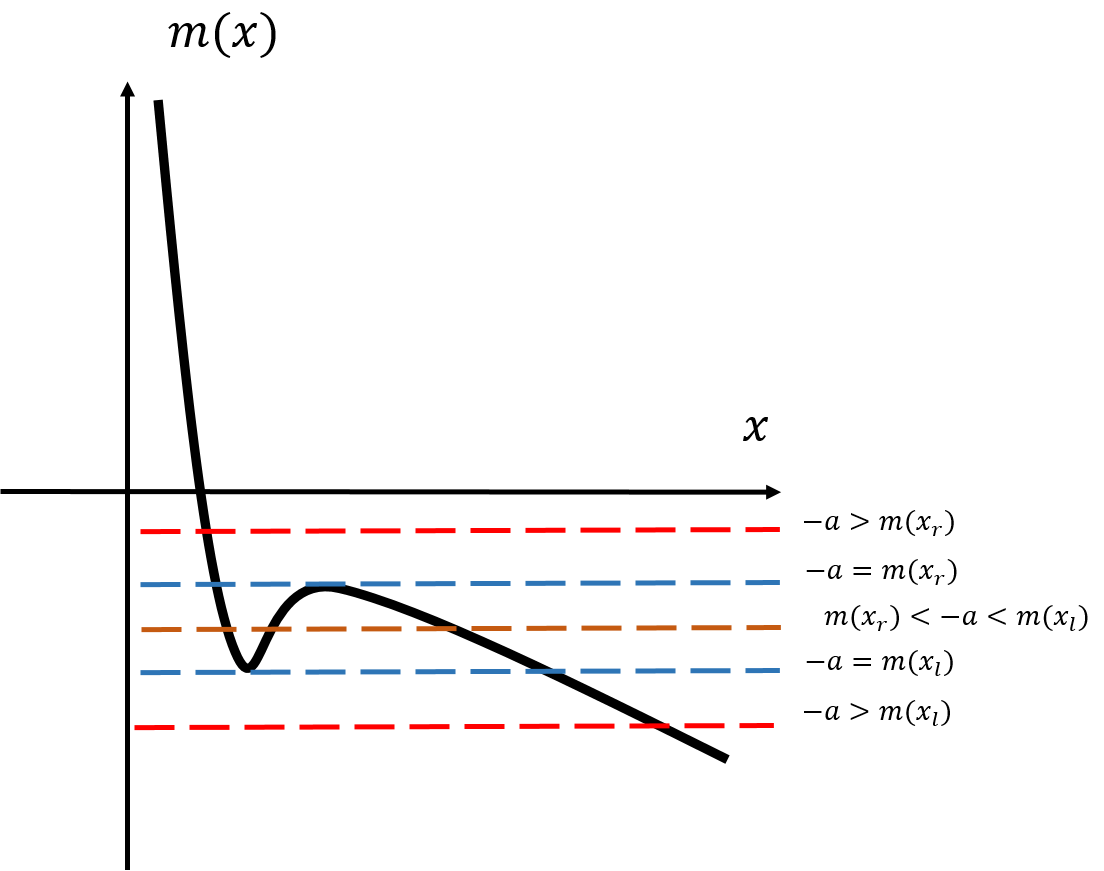}}
\caption{ The $m$ function in terms of parameters $b_1$, $b_2$.}\label{fig:funcioneme}
\end{figure}

The proof of Theorem \ref{teo:claseqv}  will be made by establishing the equilibria for system \eqref{eq:prinmod}. Let us define the energy function:

\begin{equation}
\label{eq:energia}
E(x,v):= \frac{v^2}{2}+\frac{x^2}{2}+\frac{1}{7}\dfrac{b_1}{x^7}-\frac{b_2}{x}-ax.
\end{equation} 

Note that the local minimums of $E$ correspond to non-linear centers and the local maximums correspond to saddles. However, when $ E $ has a degenerate critical point $(x^*,0)$, since the Hessian matrix $A$ is such that $\text{Tr} A = 0$, $\text{Det } A = 0$, but  $A\neq 0$. In this case, \cite{Andronov} shows, that the system  can be writen in "normal" form:
\begin{equation}
    \begin{aligned}
         \dot{x}=& y\\
         \dot{y}=& a_k x^k [1+h(x)]+b_n x^n y[1+g(x)]+y^2 R(x,y),
    \end{aligned} \label{normal}
\end{equation}
where $h(x),\,g(x)$ and $R(x,y)$ are analytic in a neighborhood of the equilibrium point $h(x^*)=g(x^*)=0$,  $k\geq 2$, $a_k \neq 0$ and $n\geq 1$. Thus the degenerate critical point $(x^*,0)$ is either a focus, a center a node, a (topological) saddle, saddle-node, a cup or a critical point with an elliptic domain, see \cite[Theorem 2, pp 151, Theorem 3, pp 151]{Perko}.

 
\begin{theorem}\label{teo:claseqv}
The equilibrium solutions of the conservative system associated with \eqref{eq:prinmod} are classified as follows:

\begin{enumerate}
\item A non-linear center if either $b_1\geq b_1^* $ and $a\in \mathbb{R}^+$ or $b_1<b_1^*$ and $a\in \{\mathbb{R}^+-]-m(x_r),-m(x_l)[\}$. 

\item Two non-linear centers and a saddle if $b_1<b_1^*$ and $a\in]-m(x_r),-m(x_l)[.$

\item A non-linear center and a cusp, if either  $b_1<b_1^*$ and $a=-m(x_r)$ or $a=-m(x_l)$.
\end{enumerate}
\end{theorem}

\begin{proof} We present here the main steps $1.-3.$ of the argument.\\
	
1. Note that $\mathcal{G}$  has a unique element if either $b_1> b_1^*$ and $a\in \mathbb{R}^+$ or $b_1<b_1^*$ and $a\in\mathbb{R}^+-]-m(x_r),-m(x_l)[$, the equilibrium is a non-linear center since  $E$ reaches a local minimum at that point. For the case $b_1=b_1^*$,  $a=-m(x_c)$  is degenerate, using the expansion given in \eqref{normal}, we have  $k=3$ and 
\[
a_k=\frac{24\,b_2}{6 \,x_c ^5}-\frac{720\,b_1^*}{6\, x_c ^{11}}<0,
\]
 therefore,from \cite[Theorem 2, pp 151]{Perko}, follows that the equilibrium is a non-linear center.\\

2. Under the hypothesis made, the set $\mathcal{G}$ has three solutions such that two are local minimums of $E$ and the other is a local maximum of $E$. Consequently, two of the equilibria  are non-linear centers and the other equilibrium is a saddle.\\

3. In this case, $\mathcal{G} $ has two solutions such that one of them is a local minimum of $E$ and corresponds to a non-linear center while the other one is degenerate with  $k=2$, $b_1=0$  in \eqref{normal}. Consequently, \cite[Theorem 3, pp 151]{Perko} guarantees that equilibrium is a cusp.
\end{proof}

In the next section,we focus on the persistence of homoclinic orbits present in  Theorem \ref{teo:claseqv} when studying the equation  \eqref{eq:prinmod}.

The conservative equation associated with \eqref{eq:prinmod} can be written as the parametric system:

\begin{equation}
\label{eq:bif}
\begin{aligned}
 	x'=&y\\
 	y'=& F(x,a),
 \end{aligned}
 \end{equation}
where $F(x,a)=m(x)+a$. Note that Theorem \ref{teo:claseqv} allows us to  build the bifurcation diagram of equilibria in terms of the parameter $a$, see figures \ref{fig:funcioneme} and \ref{fig:bif}. Moreover, when  $b_1\geq b_1^*$ the parameter $a$ does not modify the dynamics of the system  as it does when $b_1< b_1^* $. In fact, there exists  numerical evidence, see \cite{Ashhab, Mohammad}, which shows that the points  $(x_i, a_i)$,  with $a_ {i}= -m(x_{i})$, $i=r,s$ are bifurcation points. In the following theorem, it will be formally shown that those points  are saddle-node bifurcation points.

\begin{theorem}
	\label{teo:bif}
If $b_1 <b_1^* $ then the points $(x_i, a_i)$, $i=r,l$ are local saddle-node bifurcation for the conservative system  \eqref{eq:prinmod}.
\end{theorem}

\begin{proof}
 In fact, it is enough that the following conditions are fulfilled, as shown  in \cite[Theorem 3.1, pp 84]{Yuri}:
\begin{itemize}
	\item[A1] $\partial_{xx}F(x,a)|_{(x_i,a_i)}\neq 0$.
	\item[A2] $\partial_a F(x,a)|_{(x_i,a_i)}\neq 0$.
\end{itemize}	
Indeed,  we have  $\partial_{xx}F(x,a)|_{(x_l,a_l)}>0$ ( resp. $\partial_{xx}F(x,a)|_{(x_r,a_r)}<0$), because $m$ has relative minimum (resp.  maximum) in $x_l$ (resp. $x_r$) and  $\partial_a F(x,a)|_{(x_i,a_i)}=1$.
\end{proof}

To summarize, the results obtained in theorems \ref{teo:claseqv} and \ref{teo:bif} are illustrated in the bifurcation diagram of the conservative system associated to \eqref{eq:prinmod}. In part a) of Figure  \ref{fig:bif} the red curve separates the region in terms of the parameters $ b_1 $ and $ b_2 $ for which the conservative system has a unique equilibrium (independent of the parameter $a$), of the region where the number of equilibrium solutions depends on the parameter $a$. In fact, if we take $(b_2, b_1)\in\mathbb{R} ^ 2 _ + - \{(b_2, b_1)\in\mathbb{R}^2_+:b_1 \geq b_1^*\}$ then the conservative system may have one, two or three equilibria as illustrated in Figure \ref{fig:bif} (b). In this figure the solid lines are related to the stable equilibria, while the dotted line is related to the solutions of unstable equilibria. Furthermore it can be shown that locally around the points $(x_i, a_i)$, $ i=l,r$ there is a saddle-node bifurcation.

 \begin{figure}
 	\subfigure[Bifurcation Diagram in terms of the parameters $b_1$, $b_2$]{\includegraphics[scale=0.4]{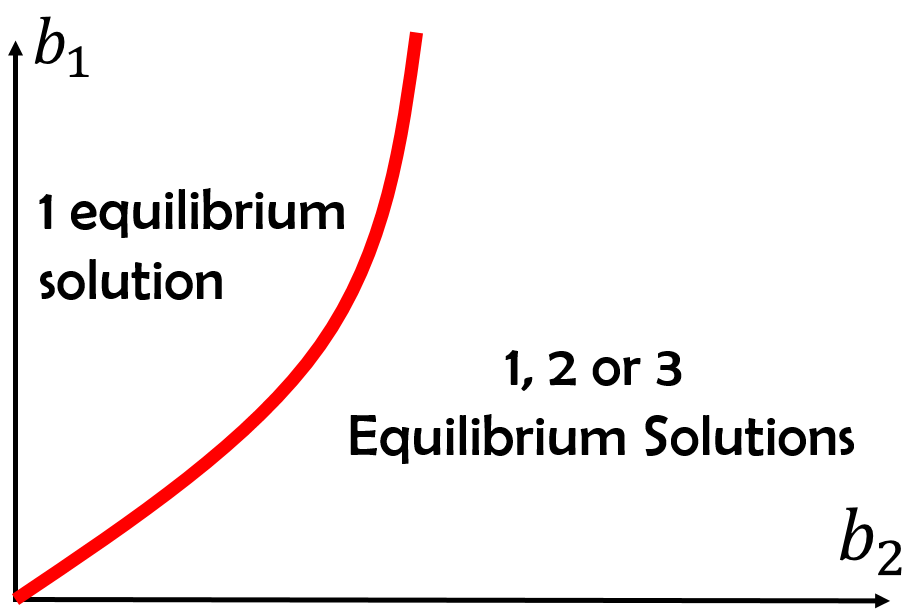}} \subfigure[Bifurcation diagram in terms of the parameter $a$ and the number of equilibrium solutions when setting $b_1$ and $b_2$ such that $b_1<b_1^*$]{\includegraphics[scale=0.4]{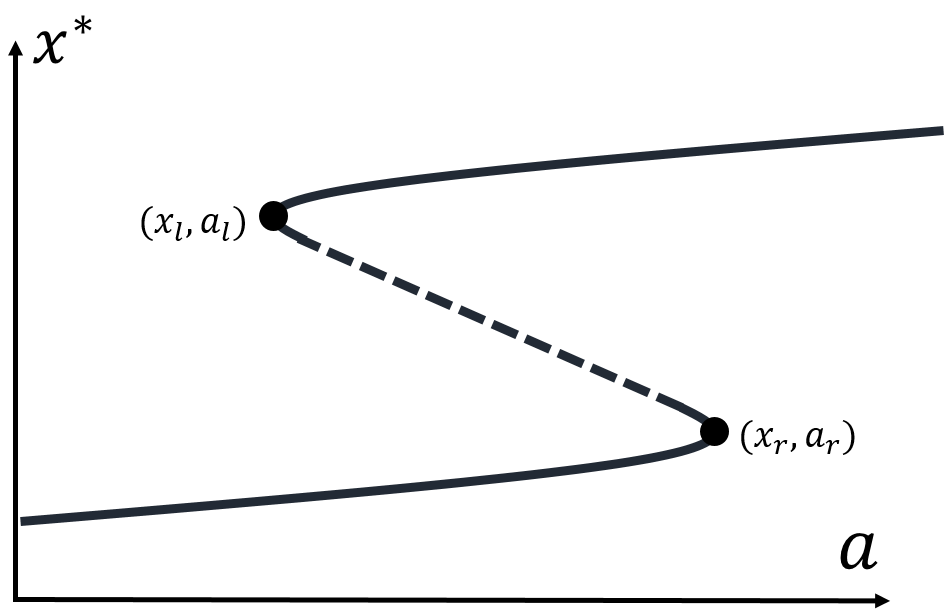}}
 	\caption{Bifurcation Diagrams of the equation \eqref{eq:prinmod} in conservative system.}\label{fig:bif}
 \end{figure}

\section{Homoclinic Persistence}

The discussion in this section is limited to the case $b_1<b_1 ^*$ and $a\in]-m(x_r), -m(x_l)[$. The objective  is to apply the Melnikov's method to \eqref{eq:prinmod} when  $f(t)=B\cos(\Omega \,t)$, it can be used to described  how the homoclinic orbits persists in the presence
of the perturbation. For AFM models the persistence of homoclinic orbits has great practical use since it can be produce uncontrollable vibrations of the device, causing fail  and  generate erroneous readings, \cite{Ashhab2, Ashhab, Amortiguamiento}. 

Before we address this problem, let us establish some notation. Consider the systems of the form

\begin{equation}
\label{eq:nearH}
x'=f(x)+\epsilon g(x,t),\quad x\in\mathbb{R}^2,
\end{equation}

where $f$ is a vector field Hamiltonian  in $\mathbb{R}^2$, $g_i\in C^{\infty}(\mathbb{R}^2\times \mathbb{R}/(T \mathbb{Z}))$, $i=1,2$, $g=(g_1,g_2)^T$ and $\epsilon\geq 0$. Now, suppose in an unperturbed system, i.e $\epsilon=0$ in \eqref{eq:nearH}, the existence of a family of periodic orbits given by 
\[
\gamma_e=\{	(x_1,x_2): E(x_1,x_2)=e\}, \quad e\in]\alpha,\beta[,
\]  	

such that $\gamma_e$ approaches a center as $e\to\alpha$ and  to an invariant curve denoted by $\gamma_\beta$, as $e\to\beta$. When $\gamma_\beta$ is bounded, it is a homoclinic loop consisting of a saddle and a connection. We want to  know if $\gamma_\beta$ persists when \eqref{eq:nearH}, where $0<\epsilon<<1 $, that is, if $\gamma_\beta(t,\epsilon)$ is a homoclinic of \eqref{eq:nearH} that is generated by $\gamma_\beta$. The first approximation of $\gamma_e (t, \epsilon) $ is given by the zeros of the Melnikov's function $M_e(t)$ defined as:
\[
\label{eq:melikov}
M_e(t):=\int_{E(x_1,x_2)=e} g_2 dx_1-g_1 dx_2,
\]	 
therefore, it is necessary to know  the number of zeros of \eqref{eq:melikov}. For our purposes, the following Theorem, which is an adaptation of \cite{Han}, will be useful.
\begin{theorem}[\cite{Han}, Theorem 6.4] \label{teo:han} Suppose $e_0\in]\alpha,\beta]$ and $t_0\in\mathbb{R}$.
\begin{enumerate}
	\item If $M_{e_0}(t_0)\neq 0$, then, there are no limit cycles near $\gamma_{e_0}$ for $\epsilon+|t_0+t|$ sufficiently small.
	
	\item 	If  $M_{e_0}(t)= 0$ is a simple zero there is exactly one limit cycle $\gamma_{{e_0}}(t_0,\epsilon)$ for $\epsilon+|t_0+t|$ sufficiently small that approaches $\gamma_{{e_0}}$ when $(t,\epsilon)\to (t_0,0)$.	 
\end{enumerate}
\end{theorem}   

\begin{remark}
Melnikov's function can be interpreted as the first approximation in $\epsilon$ of the distance between the stable and unstable manifold, measured along the direction perpendicular to the unperturbed connection, that is, $d(\epsilon):= \epsilon \frac{M_{\beta(t_0)}}{\|f(\gamma_\beta)\|}+O(\epsilon^2)$.
 In particular, when $ M_\beta(t_0)> 0 $ (resp. $ <0 $) the unstable manifold is above (resp. below) the stable manifold, see \cite{Guckenheimer, Perko} for a detail discussion.
\end{remark}

Rewriting \eqref{eq:prinmod} as a system of the form \eqref{eq:nearH}, we obtain
\begin{align*}
f(x_1,x_2)& =\begin{pmatrix}
x_2 \\ m(x_1)+a
\end{pmatrix}, & g(x_1,x_2,t) &=\begin{pmatrix}
0 \\ B \cos (\Omega \,t)-\dfrac{C}{x_1^3}x_2 
\end{pmatrix}.
\end{align*} 

From Theorem \ref{teo:claseqv}, we have that if $b_1<b_{1}^* $ and $a\in]-m(x_r),-m(x_l)[$, the unperturbed system has  three equilibria from which one is a  saddle,  denoted by $ (x_{sa}, 0) $. The  function's energy associated with the conservative system is given by \eqref{eq:energia} and homoclinic loops, denoted by $\Gamma_l$ and $\Gamma_r$, and  $E(x_1, x_2)=E(x_{sa},0)=\beta$.

When calculating Melnikov's function along the separatrix on the right $\Gamma_r$, the computation along $\Gamma_l$ is identical, this is:

\begin{align*}
M_\beta(t_0)=&\int_{\Gamma_r} g_2 dx_1 -g_1dx_2=\oint_{\gamma_{\beta_r}} (E_{x_2}g_1+E_{x_1}g_2)dt\\
=& \int_{-\infty}^{\infty} x_2(t)\left(B\, \text{cos}(\Omega (t+ t_0)
)-\frac{C}{x_1^3(t)}x_2(t)\right) dt\\
=& B\,\text{cos}(\Omega\, t_0)\int_{-\infty}^{\infty}\text{cos}(\Omega\, t)x_2(t)dt-B\,\text{sen}(\Omega\, t_0)\int_{-\infty}^{\infty}\text{sen}(\Omega\, t)x_2(t)dt \\
&- C\int_{-\infty}^{\infty}\frac{x_2^2(t)}{x_1^3(t)}dt\\
=&-2B\,\text{sen}(\Omega\, t_0)\int_{0}^{\infty}\text{sen}(\Omega\, t)x_2(t)dt - C\int_{-\infty}^{\infty}\frac{x_2^2(t)}{x_1^3(t)}dt.
\end{align*}
Note that
\[
 \int_{-\infty}^{\infty}\text{cos}(\Omega\, t)x_2(t)dt =0,
\]
due to  $\cos(\Omega \, t)x_2 (t)$ is an odd function. Consequently:
\[
M_\beta(t_0)=-2B\,\text{sen}(\Omega\, t_0)\int_{0}^{\infty}\text{sen}(\Omega\, t)x_2(t)dt - C\int_{-\infty}^{\infty}\frac{x_2^2(t)}{x_1^3(t)}dt.
\]
Define
\begin{align*}
    \xi_1&=-2\int_{0}^{\infty}\text{sen}(\Omega\, t)x_2(t)dt, &
	\xi_2&=-\int_{-\infty}^{\infty}\frac{x_2^2(t)}{x_1^3(t)}dt,  
\end{align*}
and we proof that $\xi_1$, $\xi_2$  are bounded. Indeed, $dt=dx_1/x_1 =dx_1 /x_2$ and $x_{sa}<x_1 <\bar{x}$ in $ \Gamma_r $, where $x_{sa},\,\bar{x}$ are consecutive zeros  of $E(x_1,0)-\beta$. Now if $ E(x_1, x_2)=\beta$ then 
\begin{equation*}
x_2^2=2\left(\beta+ax_1+\frac{b_2}{x_1}-\frac{b_1}{7\,x_1^7}-\frac{x_1^2}{2}\right),
\end{equation*} 
hence
\begin{align*}
\xi_1\leq  2\int_{0}^{\infty} x_2(t)dt=2\int_{x_{sa}}^{\bar{x}}dx_1=2(\bar{x}-x_{sa}).
\end{align*}
On the other hand, 
\begin{align*}
|\xi_2|\leq 2 C\int_{x_{sa}}^{\bar{x}} \left|\frac{x_2}{x_1^3}\right|dx_1=2C \int_{x_{sa}}^{\bar{x}}\frac{\sqrt{2\left(\beta+ax_1+\frac{b_2}{x_1}-\frac{b_1}{7\,x_1^7}-\frac{x_1^2}{2}\right)}}{|x_1^3|}dx_1 <\infty.
\end{align*}
Finally Melnikov's function is rewritten as
\begin{equation}
\label{eq:probmel}
M_\beta(t_0)=B\, \xi_1 \text{ sen}(\Omega\,t_0)+C\,\xi_2.
\end{equation}

\begin{theorem}\label{teo:melnivok}
Under the conditions of item 2 of the Theorem \ref{teo:claseqv} we have that the homoclinic orbits of \eqref{eq:prinmod} persist as long as $\epsilon$ is sufficiently small and:

\begin{equation}
\label{eq:cond1}
\frac{B}{C}>\bigg|\frac{\xi_2}{\xi_1}\bigg|.    
\end{equation}	
\end{theorem}

\begin{proof}

 Condition \eqref{eq:cond1} implies that  Melinikov's function \eqref{eq:probmel}  has a simple zero. Consequently, Theorem \ref{teo:han}  reaches the desired conclusion.

\end{proof}

\begin{example}
For illustrative purposes, we have taken from \cite{Rutzel} the  realistic values of the physical parameters in Table \ref{tabla1}.
\begin{table}[h]
\centering
\begin{tabular}{|c l|}
     \hline
     Symbol & Value\\
     \hline
     $A_1$ &  $0.001 X 10^-{70}$ $J m^6$  \\
     $A_2$ &  $2.96 X 10^-{19}$ $J$ \\
     $R$ &  $10$ $nm$ \\
     $K$ &  $0.87$ $N/m$\\
     $Z_0$ & $1.68108$ $nm$\\
     \hline
\end{tabular}
\caption{Properties of the case study of the AFM cantilever of Rützel et. al. \cite{Rutzel}}
\label{tabla1}
\end{table}
The values in Table \ref{tabla1} are related to the following adimensionalized values  $b_1, b_2 $ and $a$:

\begin{align*}
b_1 &=113876/10000000, & b_2 &=148148/1000000, & a&=1.07468,\\
|\xi_1| &=0.290315, & |\xi_2| &=0.382056.
\end{align*}
For instance, fix $C=1$ and $\Omega=1$, Theorem \ref{teo:melnivok} guarantees that if $ B> 1.316$ then the homoclinic persists.
\end{example}

\subsection*{Acknowledgments}
We are grateful to anonymous referees for their useful and inspiring remarks. The authors have been financially supported by the Convocatoria Interna UTP 2016, project CIE 3-17-4.

\subsection*{Data Availability}
The  data used to support the findings of this study are included within the article.

\end{document}